\documentclass[letterpaper, 10 pt, conference]{ieeeconf}

\IEEEoverridecommandlockouts    
\usepackage[utf8]{inputenc}
\usepackage{psfrag}
\usepackage{amsmath,amssymb,amsfonts,bbm,mathrsfs}
\usepackage{latexsym}
\usepackage{graphicx}
\usepackage[acronym]{glossaries}
\usepackage{colortbl}
\usepackage{fancyhdr}
\usepackage{epstopdf}
\usepackage{float}
\usepackage{pdfrender}
\usepackage{booktabs}
\usepackage{enumerate}
\usepackage[ruled]{algorithm2e}
\usepackage{balance}
\usepackage{tikz}
\usepackage{tabu}
\usepackage{cite}
\usepackage{pgfplots}
\pgfplotsset{compat=1.10}
\usepgfplotslibrary{fillbetween}

\usepackage{mathtools}

\newtheorem{theorem}{Theorem}
\newtheorem{definition}{Definition}
\newtheorem{proposition}{Proposition}
\newtheorem{lemma}{Lemma}
\newtheorem{corollary}{Corollary}

\newtheorem{remark}{Remark}

\newtheorem{standing}{Standing Assumption}
\newtheorem{problem}{Problem}

\newcommand{\R}{\mathbb{R}}

\newcommand{\N}{\mathbb{N}}

\newcommand{\mc}{\mathcal}

\newcommand{\prob}{\mathbb{P}}

\newcommand{\diag}{\mathrm{diag}}

\newcommand{\col}{\mathrm{col}}
\newcommand{\rank}{\mathrm{rank}}

\newcommand{\bs}{\boldsymbol}
\newcommand{\bsone}{\boldsymbol{1}}

\newcommand\oprocendsymbol{\hbox{$\square$}}
\newcommand\oprocend{\relax\ifmmode\else\unskip\hfill\fi\oprocendsymbol}

\graphicspath{{./figures/}}
\DeclareMathSizes{10}{10}{6}{4}

\makeglossaries
\newacronym{LP}{LP}{linear program}
\newacronym{DT}{DT}{discrete-time}
\newacronym{LTI}{LTI}{linear time-invariant}
\newglossaryentry{LMI}
{
	name={LMI},
	description={linear matrix inequality},
	first={\glsentrydesc{LMI} (\glsentrytext{LMI})},
	plural={LMIs},
	descriptionplural={linear matrix inequalities},
	firstplural={\glsentrydescplural{LMI} (LMIs)}
}
\newacronym{iid}{i.i.d.\@}{independent and identically distributed}
\newacronym{wrt}{w.r.t.\@}{with respect to}

\begin{document}

\title{Probabilistic stabilizability certificates for a class of black-box linear systems}

\author{
Filippo Fabiani, Kostas Margellos and Paul J. Goulart
\thanks{The authors are with the Department of Engineering Science, University of Oxford, OX1 3PJ, United Kingdom {\tt \footnotesize (\{filippo.fabiani, kostas.margellos, paul.goulart\}@eng.ox.ac.uk)}. This work was partially supported through the Government’s modern industrial strategy by Innovate UK, part of UK Research and Innovation, under Project LEO (Ref. 104781)
}
}

\thispagestyle{empty}
\pagestyle{empty}

\maketitle         
\begin{abstract}
We provide out-of-sample certificates on the controlled invariance property of a given set with respect to a class of black-box linear systems. Specifically, we consider linear time-invariant models whose state space matrices are known only to belong to a certain family due to a possibly inexact quantification of some parameters. By exploiting a set of realizations of those undetermined parameters, verifying the controlled invariance property of the given set amounts to a linear program, whose feasibility allows us to establish an a-posteriori probabilistic certificate on the controlled invariance property of such a set with respect to the nominal linear time-invariant dynamics. The proposed framework is applied to the control of a networked system with unknown weighted graph.
\end{abstract}


\section{Introduction}

Guaranteeing the existence of a feedback control law capable of enforcing state constraints is essential for many control systems. A well-established paradigm in the systems-and-control community requires one to verify the controlled invariance property of a certain set, thus certifying the existence of feasible control inputs that do not allow system trajectories, initialized within the set, to escape from that set (see, e.g., \cite{blanchini1999set,blanchini2015set} for a detailed discussion on the topic). 

In contrast with traditional model-based approaches, data-driven and learning techniques for control invariance and stabilizability problems have recently been attracting significant attention \cite{coppens2020data,dai2020semi,bisoffi2020controller}. Among them, a certain line of research leverages randomized methods for (controlled) invariance set estimation and set-membership verification
\cite{chen2018data,zeng2018sample,wang2019data,wang2020data,fiacchini2020probabilistic,gao2020computing,alamo2019safe,mammarella2021chance,wang2020scenario}.

Specifically, a data-driven algorithm to approximate the minimal robust control invariant set \gls{wrt} an uncertain system, albeit without invariance guarantees for unseen dynamics, was proposed in \cite{chen2018data}. 
In \cite{zeng2018sample}, the Koopman operator and the dynamic mode decomposition were used to reconstruct invariant sets for nonlinear systems by relying on a few data snapshots only. Following the same theme, data-driven methods to compute either polyhedral or maximal invariant sets with probabilistic guarantees for \gls{DT} black-box systems were presented in \cite{wang2019data,wang2020data}. 
By relying on partial knowledge of the system model, \cite{fiacchini2020probabilistic} proposed an optimization-based procedure to compute probabilistic reachable sets for linear systems affected by stochastic disturbances, while the concept of stochastic invariance for control systems through probabilistic controlled invariant sets was introduced and thoroughly investigated in~\cite{gao2020computing}. 
Randomized approaches to estimate chance-constrained sets with probabilistic guarantees, frequently encountered in control, were discussed in \cite{alamo2019safe,mammarella2021chance}. Finally, a scenario-based set invariance verification approach for black-box systems was proposed in \cite{wang2020scenario}, where the observation of system trajectory snapshots allowed to compute almost-invariant sets enjoying theoretical probabilistic invariance certificates. 

Similarly to \cite{wang2020scenario}, we investigate a scenario-based approach for the verification of the controlled invariance property of a given set. Unlike the aforementioned literature, we consider a \gls{DT} \gls{LTI} system whose nominal dynamics, described by the pair matrices $(\bar{A}, \bar{B})$, is unknown, though belonging to a certain family $\{(A(\delta), B(\delta))_{\delta \in \Delta}\}$ due to a possibly inexact quantification of some parameters, encoded by a vector $\delta \in \Delta$ (\S \ref{sec:problem_def}).
By exploiting available realizations of $\delta$, we propose a data-based affine policy to sample the space of feasible control inputs at the vertices of the given set whose controlled invariance property is to be verified. We are then able to translate the control invariance property verification of the given set with a prescribed affine policy into a \gls{LP} (\S \ref{sec:prob_cert}). The feasibility of such an \gls{LP}, along with known results in scenario theory \cite{calafiore2006scenario,campi2018general}, typically characterizing decision-making problems \cite{margellos2015connection,fabiani2020robustness,fabiani2020probabilistic,wang2020scenario}, allow us to establish an a-posteriori probabilistic bound on the controlled invariance property of a given set \gls{wrt}
any \gls{LTI} dynamics realized by unseen scenarios of $\delta$, including the nominal one (\S \ref{sec:main_res}). 
We illustrate our approach on a networked, multi-agent system with edge weights in the underlying graph not deterministically known (\S \ref{sec:motivating}, \ref{sec:num_sim}).
To the best of our knowledge, this work is the first to provide data-driven probabilistic certificates for controlled invariance verification for a class of black-box \gls{LTI} systems.

\section{Motivating example: networked multi-agent system with unknown weighted graph}\label{sec:motivating}
To motivate the control problem addressed throughout the paper, we consider a static network of $n$ entities that  exchange information locally according to a connected and undirected graph $\mc{G} \coloneqq (\mc{N}, \mc{E}, \mc{W})$ with known topology. 
The set $\mc{N} \coloneqq \{1, \ldots, n\}$ indexes the agents, which are assumed to be associated with a scalar variable $x_i \in \R$, $\mc{E} \subseteq \{(i, j) \in \mc{N}^2 \mid i \neq j\}$ denotes the information flow links, while $\mc{W} \subseteq \R^{|\mc{E}|}$ the possible weights on the edges. 
We consider an instance where the agents follow a weighted agreement protocol that is also influenced by constrained external inputs $u \in \mathscr{U} \subseteq \R^m$ injected at $m$ specific nodes. We can therefore split the set $\mc{N} = \mc{N}_F \cup \mc{N}_I$ into floating ($\mc{N}_F$, $n_F \coloneqq |\mc{N}_F|$) and input nodes ($\mc{N}_I$, $m \coloneqq |\mc{N}_I|$). 
By introducing $\col(\cdot)$ as the operator stacking its arguments in column vectors or matrices, the incidence matrix $D \in \R^{n \times |\mc{E}|}$ associated to $\mc{G}$ can be partitioned as $D = \col(D_F, D_I)$, with $D_F \in \R^{n_F \times |\mc{E}|}$ and $D_I \in \R^{m \times |\mc{E}|}$, thus leading to the following (possibly constrained) \gls{DT} \gls{LTI} dynamics characterizing the floating node states $x_F \coloneqq \col((x_i)_{i \in \mc{N}_F}) \in \mathscr{X}_F$ \cite{mesbahi2010graph}
\begin{equation}\label{eq:LTI_network}
	x_F^+ = A_F x_F + B_F u,
\end{equation}
where $A_F \coloneqq A_F(w) = I_{n_F} - D_F W D_F^\top$, $B_F \coloneqq B_F(w) = - D_F W D_I^\top$, $u \coloneqq \col((x_i)_{i \in \mc{N}_I})$ and $W \coloneqq \diag(w) \in \R^{|\mc{E}| \times |\mc{E}|}$, with $w \in \mc{W}$ a vector of weights associated with the links, whereby $\bar{w} \in \mc{W}$ characterizes their nominal values. 

Particularly when the weights are allowed to be nonpositive, the state evolutions generated by weighted consensus protocols on a network as in \eqref{eq:LTI_network} can be very rich, including steady-state trajectories that are synchronized, clustered, or even unstable \cite{zelazo2014definiteness,zelazo2015robustness}.
However, the nominal weights $\bar{w}$ on the links, especially when arising from a physical modelling, are typically hard to quantify exactly, and hence we may have available either a rough estimate of them, or some measurements \cite{zelazo2014definiteness,yucelen2015control,zelazo2015robustness,siami2017centrality}.
Therefore, within such a flexible framework, establishing whether the unknown system in \eqref{eq:LTI_network} is stabilizable by means of suitable control inputs, i.e., if the closed-loop trajectories of \eqref{eq:LTI_network} satisfy $x_F(t) \in \mc{S}_F \subseteq \mathscr{X}_F$, $t \in \N_0 \coloneqq \N \cup \{0\}$, for a given set $\mc{S}_F$, becomes instrumental. 
This essentially amounts to a control invariance problem \cite{blanchini1999set}, which for this motivating example is formalized as follows.


\begin{problem}\label{prob:main_question}
	Establish data-driven probabilistic certificates of controlled invariance for given sets \gls{wrt} black-box \gls{DT} \gls{LTI} systems as in \eqref{eq:LTI_network}, i.e., systems for which $w$ is not a-priori know, however, we have a finite number of scenarios available belonging to the set $\mc{W}$.
	\hfill$\square$
\end{problem}


\section{Problem formulation}\label{sec:problem_def}
In this paper, we consider \gls{DT} \gls{LTI} systems in the form
\begin{equation}\label{eq:linear_sys}
	x^+ = \bar{A} x + \bar{B} u,
\end{equation}
where $\bar{A} \in \R^{n \times n}$, $\bar{B} \in \R^{n \times m}$, while $x \in \mathscr{X} \subseteq \R^n$ and $u \in \mathscr{U} \subseteq \R^m$ are the constrained vectors of state variables and control inputs, respectively. Then, by defining a C-set $\mc{S} \subseteq \R^n$ as a convex, bounded set such that $0 \in \mathrm{int}(\mc{S})$, which reduces to a C-polytope if it is also polyhedral, we make the following assumption on the sets $\mathscr{X}$ and $\mathscr{U}$.
\begin{standing}
	The set $\mathscr{X}$ is a C-set, while $\mathscr{U}$ is a C-polytope.
	{\hfill$\square$}
\end{standing}

In the remainder, we use $x(t)$, $t \in \N_0$, as opposed to $x$, making the time dependence explicit, whenever necessary. 

\subsection{Stabilizability of discrete-time LTI systems}
We start by recalling the following well-known definition:

\begin{definition}\textup{(Controlled invariance)}\label{def:controlled_invariant}
	A set $\mc{S} \subseteq \mathscr{X}$ is a controlled invariant \gls{wrt} \eqref{eq:linear_sys} if there exists a $\mc{C}^1$-class feedback control law $\kappa : \R^n \to \R^m$, $\kappa(x(t)) \in \mathscr{U}$, $t \in \N_0$, such that, for any $x(0) \in \mc{S}$, the trajectory originating from \eqref{eq:linear_sys} with $u(t) = \kappa(x(t))$ satisfies $x(t) \in \mc{S}$, for all $t \in \N$.
	{\hfill$\square$}
\end{definition}

We next restate a fundamental result characterizing the controlled invariance of a C-polytope $\mc{S}$ \gls{wrt} \gls{DT} \gls{LTI} systems as in \eqref{eq:linear_sys}, which will be key in the rest of the paper.

\begin{lemma}\textup{(\hspace{-.05mm}\cite[Cor.~4.46]{blanchini2015set})}\label{lemma:control_vert}
	A C-polytope $\mc{S} \subseteq \mathscr{X}$ is controlled invariant for the \gls{DT} \gls{LTI} system in \eqref{eq:linear_sys} if and only if, for all $x_i \in \mathrm{vert}(\mc{S})$ (the set of vertices of $\mc{S}$), there exists a feasible control input $u \in \mathscr{U}$ such that $\bar{A} x_i + \bar{B} u \in \mc{S}$.
	\hfill$\square$
\end{lemma}


Verifying the controlled invariance of $\mc{S}$ therefore amounts to checking the one-step controllability at each vertex of the set. 
A commonly used feedback control law guaranteeing the stabilizability of \gls{DT} \gls{LTI} systems inside $\mc{S}$ is the piecewise vertex control law \cite{gutman1986admissible,nguyen2013implicit}. Specifically, since any state $x \in \mc{S}$ can be decomposed as $x = \sum_{i \in \mc{V}}^{}  \gamma_i x_i$ for the $N$ vertices $\{x_i\}_{i \in \mc{V}}$ of $\mc{S}$, $\mc{V} \coloneqq \{1, \ldots, N\}$, with $\bsone_N^\top \gamma \leq 1$, $\gamma_i \geq 0$, $i \in \mc{V}$, such a control law amounts to \cite[Th.~2]{nguyen2013implicit}
\begin{equation}\label{eq:vertex_law}
	u(t) = \sum_{i \in \mc{V}} \gamma^\star_i(t) u_i,
\end{equation}
where $\gamma^\star(t) \in \mathrm{argmin}_{\gamma \in [0,1]^N} \, \{\bsone_N^\top \gamma \mid \sum_{i \in \mc{V}}  \gamma_i x_i = x(t) \}$ depends on the current state $x(t)$, while $\{u_i\}_{i \in \mc{V}}$ are arbitrary admissible control values at the vertices of $\mc{S}$,  $\{x_i\}_{i \in \mc{V}}$.


In case the pair $(\bar{A}, \bar{B})$, as well as the set $\mathscr{U}$, is available, the condition in Lemma~\ref{lemma:control_vert} can be easily verified through an \gls{LP}, however, this poses several challenges if uncertainty is present on the problem data. 
In this paper, we assume the nominal pair $(\bar{A}, \bar{B})$ characterizing the behaviour of \eqref{eq:linear_sys} to be unknown, though belonging to a (possibly infinite) family of matrices parametrized by a vector $\delta \in \Delta \subseteq \R^\ell$, i.e., 
\begin{equation}\label{eq:pair_inclusion}
	(\bar{A}, \bar{B}) \in \{(A(\delta), B(\delta))_{\delta \in \Delta}\},
\end{equation}
where $A : \R^\ell \to \R^{n \times n}$ and $B : \R^\ell \to \R^{n \times m}$. Specifically, with \eqref{eq:pair_inclusion} we mean that the system in \eqref{eq:linear_sys} evolves according to a predefined \gls{DT} \gls{LTI} dynamics, which however may differ from the (unknown) nominal one due to a possibly inexact quantification of some parameters, encoded by $\delta \in \Delta$. Thus, we refer to \eqref{eq:linear_sys} as a \emph{black-box} \gls{DT} \gls{LTI} system since its nominal dynamics, which can be associated with a specific realization of $\delta$, say $\bar{\delta} \in \Delta$, it is not revealed until runtime.

\subsection{Stabilizability of \gls{LTI} systems with unknown parameters}
Since we assume that the system in \eqref{eq:linear_sys} is a black-box, we can not directly apply the control law in \eqref{eq:vertex_law} to stabilize it, despite the fact that the admissible control values at the vertices, $\bs{u} \coloneqq \col((u_i)_{i \in \mc{V}}) \in \R^{mN}$, are arbitrary in $\mathscr{U}^N$. 
Let some C-polytope $\mc{S} \subseteq \mathscr{X}$ be given. 
According to Problem~\ref{prob:main_question}, we aim at providing out-of-sample certificates on the controlled invariance of $\mc{S}$ \gls{wrt} the black-box system in \eqref{eq:linear_sys} by exploiting some observed realizations, i.e., scenarios, of the parameter $\delta$ characterizing the inclusion in \eqref{eq:pair_inclusion}. 


Formally, we assume the parameter $\delta$ to live in some probability space $(\Delta, \mc{D}, \prob)$, where $\Delta \subseteq \R^\ell$ is the support set of $\delta$, $\mc{D}$ is the associated $\sigma$-algebra and $\prob$ is a (possibly unknown) probability measure over $\mc{D}$.
We consider $\delta_K \coloneqq \{\delta^{(j)}\}_{j \in \mc{K}} = \{\delta^{(1)}, \ldots, \delta^{(K)}\} \in \Delta^K$, $\mc{K} \coloneqq \{1, \ldots, K\}$, as a finite collection of $K \in \N$ \gls{iid} realizations of $\delta$ (also called a $K$-multisample\footnote{Every $\delta_K$ is defined over the probability space $(\Delta^K, \mc{D}^K, \prob^K)$, resulting from the $K$-fold Cartesian product of the probability space $(\Delta, \mc{D}, \prob)$.}). 
We note that to any realization $\delta \in \Delta$ is associated a pair of matrices $(A(\delta), B(\delta))$. Thus, we define the set of admissible control values at the vertices $\{x_i\}_{i \in \mc{V}}$ of $\mc{S}$ allowed by such a realization $\delta \in \Delta$ as
\begin{equation}\label{eq:adm_set}
	\mc{U}_\delta \coloneqq \{ \bs{u} \in \mathscr{U}^N \mid A(\delta) x_i + B(\delta) u_i \in \mc{S}, \, \forall i \in \mc{V} \}.
\end{equation}
According to Lemma~\ref{lemma:control_vert}, as long as $\mc{U}_\delta \neq \emptyset$, the set $\mc{S}$ is a controlled invariant set for $x^+ = A(\delta) x + B(\delta) u$, which is stabilizable by means of the piecewise vertex control law \eqref{eq:vertex_law} with admissible control values contained in $\mc{U}_{\delta}$. Thus, aiming to establish controlled invariance certificates to previously unseen realizations of $\delta$, we introduce the definition of \emph{violation probability} for a generic vector of input values $\bs{u}$.

\begin{definition}\textup{(Violation probability)}\label{def:viol_prob}
	The violation probability associated with the input values $\bs{u} \in \mathscr{U}^N$ is given~by
	\begin{equation}\label{eq:viol_prob}
		V(\bs{u}) \coloneqq \mathbb{P}\{\delta \in \Delta \mid \bs{u} \notin \mc{U}_\delta\}.
	\end{equation}
	\hfill$\square$
\end{definition}

According to Lemma~\ref{lemma:control_vert}, $V : \mathscr{U}^N \to [0,1]$ measures the violation of the controlled invariance of the set $\mc{S}$ associated with input values $\bs{u}$ \gls{wrt} an unseen pair $(A(\delta), B(\delta))$. 
In other words, $V(\bs{u})$ measures the realizations $\delta \in \Delta$ such that, when these are drawn, $\bs{u}$ can not guarantee the controlled invariance of $\mc{S}$ \gls{wrt} the system induced by $(A(\delta), B(\delta))$.


\section{Dealing with the uncertainty}\label{sec:prob_cert}
Note that, given any $K$-multisample $\delta_K \in \Delta^K$, solving an \gls{LP} allows us to compute some vector of input values $\bs{u}^{\star}_{K} \in \mc{U}_{\delta_K} \coloneqq \cap_{j \in \mc{K}} \, \mc{U}_{\delta^{(j)}}$ such that $\bs{u}^{\star}_{K} \in \mathscr{U}^N$, and
\begin{equation}\label{eq:implication_same_u}
	\forall j \in \mc{K} : \, A(\delta^{(j)}) x_i + B(\delta^{(j)}) u^{\star}_{i, K} \in \mc{S}, \, \forall i \in \mc{V}.
\end{equation}

We therefore wish to establish an a-posteriori bound on the violation probability $V(\bs{u}^{\star}_{K})$, to claim with high confidence that the probability $\bs{u}^{\star}_{K}$ guarantees the controlled invariance of $\mc{S}$ \gls{wrt} the family $\{(A(\delta^{(j)}), B(\delta^{(j)}))_{j \in \mc{K}}\} \cup \{(A(\delta), B(\delta))\}$ is above a certain value. By Lemma~\ref{lemma:control_vert}, this is equivalent to concluding that $\mc{S}$ is a controlled invariant for the black-box system in \eqref{eq:linear_sys}, with the same high confidence.

\subsection{General control policies}\label{subsec:sampling_pol}
Note the conservatism inherent in \eqref{eq:implication_same_u}. For any vertex $i \in \mc{V}$, one would consider exactly the same admissible input value, $u^\star_{i,K}$, for all the observed $K$ samples.
To alleviate this conservativism, we introduce a policy for each vertex, namely some (possibly multi-valued) functional $\pi_i : \Delta \to \R^{m}$, which maps any realization $\delta \in \Delta$ to some input value in $\R^{m}$. In fact, according to Lemma~\ref{lemma:control_vert}, on each vertex it suffices to find an admissible
control value for every observed scenario $\delta^{(j)}$, i.e., for every pair of matrices $(A(\delta^{(j)}), B(\delta^{(j)}))$, $j \in \mc{K}$.
Given a generic sample $\delta \in \Delta$, let $\Pi_\delta \coloneqq \{ \bs{\pi} : \Delta \to \R^{mN} \mid \pi_i(\delta) \in \mc{U}, A(\delta) x_i + B(\delta) \pi_i(\delta) \in \mc{S}, \text{ for all } i \in \mc{V}\}$ be the set of mappings $\bs{\pi}(\cdot) \coloneqq \col((\pi_i(\cdot))_{i \in \mc{V}})$ returning admissible inputs at the vertices of $\mc{S}$ for the pair $(A(\delta), B(\delta))$.
Note in addition that $\Pi_\delta \neq \emptyset$ guarantees that $\mc{S}$ is controlled invariant for the \gls{DT} \gls{LTI} system described by the specific matrices $(A(\delta), B(\delta))$ associated with the scenario $\delta$. 
However, looking for an element in $\Pi_\delta$ amounts to an infinite dimensional problem, as such a set contains all possible mappings $\bs{\pi}(\cdot)$. 


\subsection{Affine control policies}
To make the problem computationally tractable, we focus on a family of mappings with finite parametrization, i.e., the affine ones. 
Thus, for each vertex $i \in \mc{V}$, we define $\pi_i(\delta) \coloneqq C_i \delta + d_i$, with $C_i \in \R^{m \times \ell}$ and $d_i \in \R^{m}$, which leads to $\bs{\pi}(\delta) \coloneqq \bs{C} \delta + \bs{d}$, where $\bs{C} \coloneqq \col((C_i)_{i \in \mc{V}})$ and $\bs{d} \coloneqq \col((d_i)_{i \in \mc{V}})$ belong to $\mc{M} \coloneqq \{(\bs{C}, \bs{d}) \mid C_i \in \R^{m \times \ell}, d_i \in \R^m\}$. 
Then, the set of admissible affine policies for a given $\delta \in \Delta$ is $\mc{L}_\delta \coloneqq \{ (\bs{C}, \,\bs{d}) \in \mc{M} \mid C_i \delta + d_i \in \mathscr{U}, A(\delta) x_i + B(\delta) (C_i \delta + d_i) \in \mc{S}, \, \forall i \in \mc{V}\} \subset \Pi_\delta$.
The fact that $\mc{L}_\delta \neq \emptyset$ ensures that $\mc{S}$ is a controlled invariant for the system induced by $(A(\delta), B(\delta))$.
Given $K$ observations $\delta_K \in \Delta^K$, an optimal affine policy $(\bs{C}^\star_{K}, \bs{d}^\star_{K}) \in \mc{L}_{\delta_K} \coloneqq \cap_{j \in \mc{K}} \, \mc{L}_{\delta^{(j)}}$ satisfies, for all $j \in \mc{K}, \bs{C}^\star_{K} \delta^{(j)} + \bs{d}^\star_{K} \in \mathscr{U}^N$, and
\begin{equation}\label{eq:implication_linear_policy}
	A(\delta^{(j)}) x_i + B(\delta^{(j)}) (C^\star_{i, K} \delta^{(j)} + d^\star_{i, K}) \in \mc{S}, \forall i \in \mc{V}.
\end{equation}

For any vertex $i \in \mc{V}$ we now obtain a different admissible input value depending on the sample $\delta^{(j)}$ at hand, in contrast with the conservative approach in \eqref{eq:implication_same_u}. 
Moreover, unlike the infinite dimensional problem introduced in \S \ref{subsec:sampling_pol}, computing a pair $(\bs{C}^\star_{K}, \bs{d}^\star_{K})$ amounts to finding a feasible solution to an \gls{LP}. The C-polytopes $\mc{S}$ and $\mathscr{U}$ are $\mc{S} \coloneqq \{x \in \R^n \mid F x \leq \bsone_p \}$ and $\mathscr{U} \coloneqq \{ u \in \R^m \mid H u \leq \bsone_q \}$, where $F \in \R^{p \times n}$ and $H \in \R^{q \times m}$ have full column rank \cite[\S 3.3]{blanchini2015set}. Manipulating the inclusions in \eqref{eq:implication_linear_policy} with $\bs{x} \coloneqq \col((x_i)_{i \in \mc{V}})$ leads directly to
\begin{equation}\label{eq:linear_prog_affine}
	\mc{L}_{\delta_K} =  \underset{j \in \mc{K}}{\cap} \, \underset{\bs{C}, \bs{d}}{\mathrm{argmin}} \, \{ 0 \mid 
	G(\delta^{(j)}) (\bs{C} \delta^{(j)} + \bs{d}) \leq l(\delta^{(j)})
	\},
\end{equation}
where $G(\delta^{(j)}) \coloneqq \col(H\otimes I_N, F B(\delta^{(j)})\otimes I_N)$ and $l(\delta^{(j)}) \coloneqq \col(\bsone_{qN}, \bsone_{pN} - (F A(\delta^{(j)}) \otimes I_N) \bs{x})$. It follows that, for every $j \in \mc{K}$, \eqref{eq:linear_prog_affine} amounts to a feasibility problem defined over $(q+p)N$ linear constraints characterized by matrices $F$, $H$ and the pair matrices $\{(A(\delta^{(j)}), B(\delta^{(j)}))_{j \in \mc{K}}\}$. Via standard manipulations, $\mc{L}_{\delta_K}$ can be compactly rewritten as in \eqref{eq:LP_complete}.

\begin{figure*}[!h]
	\begin{equation}\label{eq:LP_complete}
		\mc{L}_{\delta_K} = \underset{\bs{C}, \bs{d}}{\mathrm{argmin}} \left\{ 0 \; \middle\rvert \;
		\diag((G(\delta^{(j)}))_{j \in \mc{K}}) (	(\bs{C} \otimes I_K) \col((\delta^{(j)})_{j \in \mc{K}}) + \bs{d} \otimes \bsone_K  ) \leq \col((l(\delta^{(j)}))_{j \in \mc{K}})
		\right\}.
	\end{equation}
	\hrulefill
	\vspace*{-12pt}
\end{figure*}

\section{A-posteriori probabilistic certificates of controlled invariance}\label{sec:main_res}

\subsection{Main result}
In case $\mc{L}_{\delta_K} \neq \emptyset$, an optimal pair $(\bs{C}^\star_{K}, \bs{d}^\star_{K})$ may not be unique since \eqref{eq:LP_complete} is a feasibility problem. We henceforward assume that a tie-break rule guaranteeing the uniqueness of the solution to \eqref{eq:LP_complete} is in place. 
This allows us to introduce a single-valued mapping $\Theta_K : \Delta^K \to \mc{M}$ that, given any $\delta_K \in \Delta^K$, satisfies
$
\Theta_K(\delta_K) \coloneqq (\bs{C}^\star_{K}, \bs{d}^\star_{K}).
$
We next recall the key definition of \emph{support subsample} to establish our probabilistic certificate of controlled invariance for a given C-polytope $\mc{S}$. 

\begin{definition}\textup{(Support subsample, \cite[Def.~2]{campi2018general})}\label{def:support_sub}
	Given any $\delta_K \in \Delta^K$, a support subsample $S \subseteq \delta_K$ is a $p$-tuple of unique elements of $\delta_K$, $S \coloneqq \{\delta^{(i_1)}, \ldots, \delta^{(i_p)}\}$, with $i_1 < \ldots < i_p$, that gives the same solution as the original $K$-multisample, i.e.,
	$
	\Theta_p(\delta^{(i_1)}, \ldots, \delta^{(i_p)}) = \Theta_K(\delta^{(1)}, \ldots, \delta^{(K)}).
	$
	\hfill$\square$
\end{definition}

Then, let $\Upsilon_K : \Delta^K \rightrightarrows \mc{K}$ be any algorithm returning a $p$-tuple ${i_1, \ldots, i_p}$, $i_1 < \ldots < i_p$, such that $\{\delta^{(i_1)}, \ldots, \delta^{(i_p)}\}$ is a support subsample for $\delta_K$, and let $s_K \coloneqq |\Upsilon_K(\delta_K)|$. In this case, a support subsample for $\delta_K$ can be identified as the subset of samples that generates a minimal representation for the polyhedral feasible set of \eqref{eq:LP_complete}. The following result characterizes the violation probability of the optimal pair $(\bs{C}^\star_{K}, \bs{d}^\star_{K})$, and therefore establishes a probabilistic certificate for the controlled invariance property of the set $\mc{S}$ \gls{wrt} the black-box linear system in \eqref{eq:linear_sys}, thus addressing Problem~\ref{prob:main_question}.

\begin{theorem}\label{th:prob_cert}
	Fix $\beta \in (0,1)$ and, for any $K \in \N$, let $\varepsilon : \mc{K} \cup \{0\} \to [0, 1]$ be a function such that $\varepsilon(K) = 1$ and $\sum_{h = 0}^{K - 1} \left( \begin{smallmatrix}
			K\\
			h
	\end{smallmatrix} \right) (1 - \varepsilon(h))^{K - h} = \beta$.
	Given any C-polytope $\mc{S} \subseteq \mathscr{X}$, $K$-multisample $\delta_K \in \Delta^K$ with associated matrices $\{(A(\delta^{(j)}), B(\delta^{(j)}))_{j \in \mc{K}}\}$, assume that $\mc{L}_{\delta_K}$ in \eqref{eq:LP_complete} is nonempty. Then, for any $\Theta_K$, $\Upsilon_K$ and $\mathbb{P}$, it holds that
	\begin{equation}\label{eq:prob_feas_boud}
		\mathbb{P}^K \{\delta_K \in \Delta^K \mid V(\bs{C}^\star_{K} \delta + \bs{d}^\star_{K}) > \varepsilon(s_K) \} \leq \beta,
	\end{equation}
	namely, the probability that $\mc{S}$ is a controlled invariant set \gls{wrt} the black-box system in \eqref{eq:linear_sys} is at least $1 - \varepsilon(s_K)$ with confidence greater than or equal to $1-\beta$.
	\hfill$\square$
\end{theorem}
\begin{proof}
	Given any polyhedral C-set $\mc{S}$ and $K$-multisample $\delta_K \in \Delta^K$ with associated pairs of matrices $\{(A(\delta^{(j)}), B(\delta^{(j)}))_{j \in \mc{K}}\}$, assuming that $\mc{L}_{\delta_K} \neq \emptyset$ implies that an optimal pair $(\bs{C}^\star_{K}, \bs{d}^\star_{K})$ solving \eqref{eq:LP_complete} exists and, assuming some tie-break rule, it is also unique. Therefore, we have $(\bs{C}^\star_{K}, \bs{d}^\star_{K}) \in \mc{L}_{\delta_K}$, which clearly entails the inclusion $(\bs{C}^\star_{K}, \bs{d}^\star_{K}) \in \mc{L}_{\delta^{(j)}}$, for all $j \in \mc{K}$. By construction, this means that, for every $\delta^{(j)} \in \delta_K$, $\bs{C}^\star_{K} \delta^{(j)} + \bs{d}^\star_{K} \in \mc{U}_{\delta_K}$ (see \eqref{eq:adm_set}), and hence that $\bs{C}^\star_{K} \delta^{(j)} + \bs{d}^\star_{K} \in \mc{U}_{\delta^{(j)}}$, for all $j \in \mc{K}$. These inclusions correspond to the so-called \emph{consistency condition} stated in \cite[Ass.~1]{campi2018general} and, together with the uniqueness of the solution, we can rely on \cite[Th.~1]{campi2018general} to obtain the probabilistic bound in \eqref{eq:prob_feas_boud}, i.e., $\mathbb{P}^K \{\delta_K \in \Delta^K \mid V(\bs{C}^\star_{K} \delta + \bs{d}^\star_{K}) > \varepsilon(s_K) \} \leq \beta.$
	In view of Lemma~\ref{lemma:control_vert}, we recall that $(\bs{C}^\star_{K}, \bs{d}^\star_{K}) \in \mc{L}_{\delta_K}$ is a necessary and sufficient condition for the affine sampling policy to return feasible input values guaranteeing the controlled invariance property of $\mc{S}$ \gls{wrt} the observed collection of \gls{DT} \gls{LTI} systems originated by the pairs $\{(A(\delta^{(j)}), B(\delta^{(j)}))_{j \in \mc{K}}\} \subseteq \{(A(\delta), B(\delta))_{\delta \in \Delta}\}$, since $A(\delta^{(j)}) x_i + B(\delta^{(j)})(\bs{C}^\star_{K} \delta^{(j)} + \bs{d}^\star_{K}) \in \mc{S}$, for all $i \in \mc{V}$, $j \in \mc{K}$. 
	Thus, the bound in \eqref{eq:prob_feas_boud} certifies that, with confidence at least $1-\beta$, $V(\bs{C}^\star_{K} \delta + \bs{d}^\star_{K}) = \prob\{\delta \in \Delta \mid \bs{C}^\star_{K} \delta + \bs{d}^\star_{K} \notin \mc{U}_\delta \} \leq \varepsilon(s_K)$, and therefore it turns out that $\prob\{\delta \in \Delta \mid \bs{C}^\star_{K} \delta + \bs{d}^\star_{K} \in \mc{U}_\delta \} \geq 1 - \varepsilon(s_K)$ with the same (arbitrarily high) confidence. Again, in view of Lemma~\ref{lemma:control_vert}, this means that the affine policy computed in \eqref{eq:LP_complete} returns feasible input values at the vertices of $\mc{S}$ that guarantee the controlled invariance property of $\mc{S}$ \gls{wrt} the \gls{DT} \gls{LTI} system originated by the pair of matrices $(A(\delta), B(\delta))$ associated to any possible unseen scenario $\delta \in \Delta$,
	and hence concludes the proof.
\end{proof}
\begin{remark}\label{remark:generality}
	Note that Theorem~\ref{th:prob_cert} certifies the controlled invariance property of some C-polytope $\mc{S}$ not only \gls{wrt} the nominal dynamics of \eqref{eq:linear_sys} generated by $(\bar{A}, \bar{B})$, but \gls{wrt} any \gls{DT} \gls{LTI} system originated by the pair of matrices $(A(\delta), B(\delta))$ associated to an unseen scenario of $\delta$.
	\hfill$\square$
\end{remark}

The following result, which follows directly from Theorem~\ref{th:prob_cert}, allows us to characterize in terms of probabilistic stabilizability guarantees the vertex control law in \eqref{eq:vertex_law}.

\begin{corollary}\label{cor:prob_cert_law}
	Under the same conditions of Theorem~\ref{th:prob_cert}, the probability that the vertex control law in \eqref{eq:vertex_law}, with input at vertices $\{C^\star_{i,K} \bar{\delta} + d^\star_{i,K}\}_{i \in \mc{V}}$, makes the given C-polytope $\mc{S}$ controlled invariant \gls{wrt} the \gls{DT} \gls{LTI} system in \eqref{eq:linear_sys} is at least $1 - \varepsilon(s_K)$ with confidence greater than or equal to $1-\beta$.
	\hfill$\square$
\end{corollary}
\begin{proof}
	From Theorem~\ref{th:prob_cert}, if \eqref{eq:LP_complete} is feasible, then for any unobserved sample $\delta \in \Delta$, the probability that the affine policy $\bs{\pi}(\delta) = \bs{C}^\star_{K} \delta + \bs{d}^\star_{K}$ returns admissible control values at the vertices of $\mc{S}$ is at least $1 - \varepsilon(s_K)$ with confidence $1-\beta$. Therefore, with the same confidence,
	$
		u(t) = \sum_{i \in \mc{V}} \gamma_i^\star(t) (C^\star_{i,K} \bar{\delta} + d^\star_{i,K}),
	$
	with $\gamma^\star(t) \in \mathrm{argmin}_{\gamma \in [0,1]^N} \, \{\bsone_N^\top \gamma \mid \sum_{i \in \mc{V}}  \gamma_i x_i = x(t) \}$, stabilizes the system \eqref{eq:linear_sys} with at least the same probability $1 - \varepsilon(s_K)$.
\end{proof}

Following Remark~\ref{remark:generality}, the vertex control law enjoys the stabilizability guarantees in Corollary~\ref{cor:prob_cert_law} for any unobserved sample $\delta \in \Delta$, i.e., with confidence at least $1-\beta$ and input at vertices $\{C^\star_{i,K} \delta + d^\star_{i,K}\}_{i \in \mc{V}}$, the control in \eqref{eq:vertex_law} stabilizes $x^+ = A(\delta) x + B(\delta) u$ with probability at least $1 - \varepsilon(s_K)$.

\begin{remark}
	At computation time, i.e., when solving \eqref{eq:LP_complete}, the unobserved sample $\delta \in \Delta$ is not known, while at runtime of the control law in \eqref{eq:vertex_law} it is available and time-invariant.
	\hfill$\square$
\end{remark}

\subsection{On the nonemptiness of $\mc{L}_{\delta_K}$}

The probabilistic certificate in Theorem~\ref{th:prob_cert} and, specifically, the bound in \eqref{eq:prob_feas_boud}, strongly depends on the feasibility of each \gls{LP} in \eqref{eq:linear_prog_affine}, namely $\mc{L}_{\delta_K} \neq \emptyset$.
However, in case the data matrices at hand can not guarantee the nonemptiness of $\mc{L}_{\delta_K}$, we can not conclude on the controlled invariance property of $\mc{S}$ \gls{wrt} the black-box system in \eqref{eq:linear_sys}. In fact, the \gls{LP} in \eqref{eq:LP_complete} builds upon a specific choice for the sampling policy, i.e., the affine one $\bs{\pi}(\delta) = \bs{C} \delta + \bs{d}$, and this allows us to explore only a portion of the space of feasible control values $\mathscr{U}^{NK}$. 

We characterize next the feasibility of the \gls{LP} in \eqref{eq:linear_prog_affine}, obtained with $K = 1$, in terms of problem data. To simplify notation, in the statement and related proof, we omit the dependency on $\delta$ in $G$ and $l$. In what follows, we denote with $(P)_i$ (resp., $y_i$) the $i$-th row (element) of a generic matrix (vector) $P\in\R^{n\times m}$ ($y \in \R^n$). Given a set of indices $\mc{I} \subseteq \{1, \ldots, n\}$, we indicate with $P_\mc{I}$ (resp., $y_\mc{I}$) a submatrix (subvector) obtained by selecting the rows (elements) in $\mc{I}$. 

\begin{lemma}\label{lemma:singleLP_feasibility}
	Let $n \geq m$, $K = 1$ and $\delta \in \Delta$ be any given sample with associated pair of matrices $(A(\delta), B(\delta))$. Given any C-polytope $\mc{S} \subseteq \mathscr{X}$, the set $\mc{L}_\delta$ in \eqref{eq:linear_prog_affine} is nonempty if and only if, for all $i \in \mc{V}$, there exists an invertible submatrix $G_{\mc{Q} \cup \mc{P}} \in \R^{m \times m}$ of $G$, with row indices $\mc{Q} \subset \{1,\ldots,q\}$, $\mc{P} \subset \{1,\ldots,p\}$, and related subvector $l_{\mc{Q} \cup \mc{P}}$ of $l$, such that
	\begin{equation}\label{eq:iff_conditions}
		\left\{
		\begin{aligned}
			& (H)_k G_{\mc{Q} \cup \mc{P}}^{-1} l_{\mc{Q} \cup \mc{P}} \leq 1, \, \forall k \in \bar{\mc{Q}}\\
			& (FB(\delta))_k G_{\mc{Q} \cup \mc{P}}^{-1} l_{\mc{Q} \cup \mc{P}} \leq 1 - (FA(\delta)x_i)_k, \, \forall k \in \bar{\mc{P}}
		\end{aligned}
		\right.
	\end{equation}
	with $\bar{\mc{Q}} \coloneqq \{1,\ldots,q\} \setminus \mc{Q}$, and $\bar{\mc{P}} \coloneqq \{1,\ldots,p\} \setminus \mc{P}$.
	\hfill$\square$
\end{lemma}
\begin{proof}
	See Appendix.
\end{proof}

Extending the conditions established in Lemma~\ref{lemma:singleLP_feasibility} to the general case of $K \in \N$ is, however, nontrivial. In fact, from \eqref{eq:LP_complete} it is evident that, with the same pair $(\bs{C}, \bs{d})$, one has to satisfy the inequality $G(\delta^{(j)}) (\bs{C} \delta^{(j)} + \bs{d}) \leq l(\delta^{(j)})$ for all $j \in \mc{K}$, and therefore $(\bs{C} \otimes I_K) \col((\delta^{(j)})_{j \in \mc{K}}) + \bs{d} \otimes \bsone_K \in \prod_{j \in \mc{K}}^{} \mc{U}_{\delta^{(j)}} \subseteq \mathscr{U}^{NK}$. In terms of data matrices, the following statement provides necessary conditions only for the existence of an optimal pair $(\bs{C}^\star_{K}, \bs{d}^\star_{K}) \in \mc{L}_{\delta_K}$, for some $K \in \N$, as they essentially guarantee that $\prod_{j \in \mc{K}}^{} \mc{U}_{\delta^{(j)}} \neq \emptyset$.

\begin{proposition}\label{prop:feasibility_LP}
	Let $n \geq m$, $K \in \N$ and $\delta_K \in \Delta^K$ be any given $K$-multisample with associated pairs of matrices $\{(A(\delta^{(j)}), B(\delta^{(j)}))_{j \in \mc{K}}\}$. Given any C-polytope $\mc{S} \subseteq \mathscr{X}$, the set $\mc{L}_{\delta_K}$ in \eqref{eq:LP_complete} is nonempty only if, for all $(i,j) \in \mc{V} \times \mc{K}$, there exists an invertible submatrix $G_{\mc{Q} \cup \mc{P}} \in \R^{m \times m}$ of $G(\delta^{(j)})$, with row indices as in Lemma~\ref{lemma:singleLP_feasibility}, and subvector $l_{\mc{Q} \cup \mc{P}}$ of $l(\delta^{(j)})$, satisfying the conditions in \eqref{eq:iff_conditions}.
	\hfill$\square$
\end{proposition}
\begin{proof}
	See Appendix.
\end{proof}

Proposition~\ref{prop:feasibility_LP} is only necessary for $\mc{L}_{\delta_K} \neq \emptyset$. In fact, if some $\bs{\tilde{u}} \coloneqq \col((\bs{u}_j)_{j \in \mc{K}}) \in \R^{mNK}$ satisfying $\diag((G(\delta^{(j)}))_{j \in \mc{K}}) \bs{\tilde{u}} \leq \col((l(\delta^{(j)}))_{j \in \mc{K}})$ exists, say $\bs{\tilde{u}}^\star$, then this does not imply that we are able to find a pair $(\bs{C}^\star_{K}, \bs{d}^\star_{K})$ such that $(\bs{C}^\star_{K}\otimes I_K) \col((\delta^{(j)})_{j \in \mc{K}}) + \bs{d}^\star_{K} \otimes \bsone_K = \bs{\tilde{u}}^\star$. 

To generalize Theorem~\ref{th:prob_cert} to the case where $\mc{L}_{\delta_K}$ might be empty and $K \geq 1$, let $\mc{F}_K \coloneqq \{\delta_K \in \Delta^K \mid \mc{L}_{\delta_K} \neq \emptyset\}$. The bound in \eqref{eq:prob_feas_boud} holds then with $\mc{F}_K$ in place of $\Delta^K$, implying that, if the resulting \gls{LP} is feasible, then the probability of violation is at least $\varepsilon(s_K)$ with confidence at most $\beta$ \cite{calafiore2006scenario,margellos2015connection}. In other words, $\mc{F}_K$ is the restriction of $\Delta^K$ to the $K$-multisamples for which the \gls{LP} in \eqref{eq:LP_complete} is feasible.

\section{Motivating example revisited}\label{sec:num_sim}

\begin{figure}[t!]
	\centering
	\includegraphics[width=.93\columnwidth]{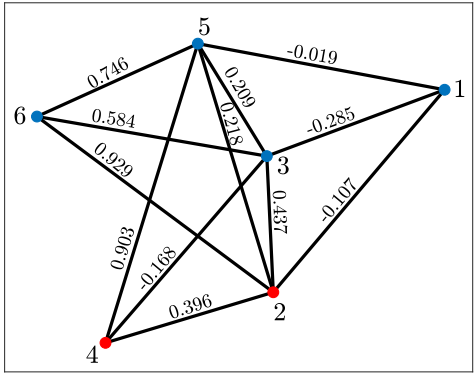}
	\caption{Graph topology with nominal weights on the edges (black lines). The blue dots denote the floating nodes, while the red dots the input ones.}
	\label{fig:graph_top}
\end{figure}

We illustrate our findings on a numerical instance of the control problem introduced in \S \ref{sec:motivating}. Specifically,
we consider the graph topology represented in Fig~\ref{fig:graph_top}, involving $n = 6$ agents, with $\mc{N}_F = \{1,3,5,6\}$,  $\mc{N}_I = \{2,4\}$, and $|\mc{E}| = 12$ edges with nominal weights $\bar{w}$ specified on each link. In this case, the autonomous dynamics in \eqref{eq:LTI_network} associated to the $n_F = 4$ floating nodes (i.e., with $B_F = 0$) is characterized by $\mathrm{eig}(A_F) = \{-1.34, -0.01, 0.46, 0.81\}$, hence unstable. Additionally, we constraint $m = 2$ control inputs to the set $\mathscr{U} = \{u \in \R^2 \mid \|u\|_{\infty} \leq 1\}$. For simplicity, $\mc{S}$ is taken as the convex hull of random points in $\pm [0.1 ~ 2]$, sampled individually on each axis of $\R^4$, leading to a C-polytope with $N = 8$ vertices.
By assuming that the entire vector of weights is not deterministically known, i.e., $\ell = |\mc{E}|$, we treat $w$ as a random vector and draw $K = 600$ samples according to a uniform distribution supported on $\mc{W} = [0.6 ~ 1.4] \times \bar{w} \subset \R^{12}$, i.e., a degree of uncertainty on $\bar{w}$ up to the $40\%$, and we compute an optimal pair $(\bs{C}^{\star}_{600}, \bs{d}^{\star}_{600})$ by solving \eqref{eq:LP_complete} with cost function $\|\bs{C}\|^2_F + \|\bs{d}\|^2$. 
The greedy algorithm designed in \cite[\S II]{campi2018general} returns a support subsample of cardinality $s_{600} = 29$, and therefore, with $\beta = 10^{-6}$, from Theorem~\ref{th:prob_cert} the probability that $\mc{S}$ is a controlled invariant for the floating dynamics in \eqref{eq:LTI_network} is at least $0.7911$, with confidence greater than or equal to $1 - 10^{-6}$. 
The function $\varepsilon(\cdot)$ in Theorem~\ref{th:prob_cert} is analytically obtained by splitting $\beta$ evenly among the $600$ terms within the summation, thus obtaining $\varepsilon(29) = 0.2089$. 

\begin{figure}[t!]
	\centering
	\includegraphics[width=.93\columnwidth]{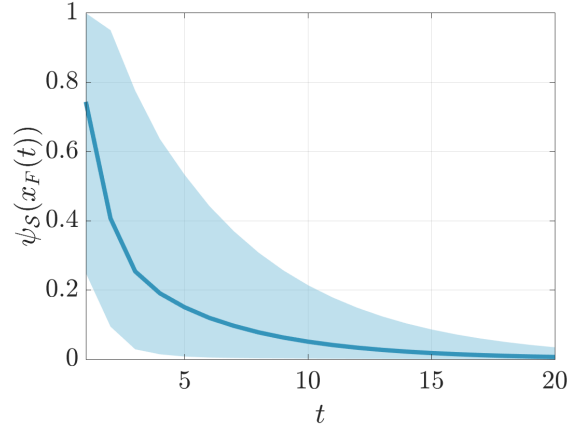}
	\caption{Time evolution of the Minkowski function associated to the C-polytope $\mc{S}$, obtained from each closed-loop trajectory of the nominal dynamics in \eqref{eq:LTI_network} originating from $10^3$ randomly chosen initial conditions inside $\mc{S}$, and control law in \eqref{eq:vertex_law} with admissible inputs $\{C^\star_{i,600} \bar{w} + d^\star_{i,600}\}_{i \in \mc{V}}$.}
	\label{fig:min_fun}
\end{figure}


Moreover, according to Corollary~\ref{cor:prob_cert_law}, the vertex control law in \eqref{eq:vertex_law}, with input at the vertices $\{C^\star_{i,600} \bar{w} + d^\star_{i,600}\}_{i \in \mc{V}}$, also enjoys the same probability certificate of $\mc{S}$. Figure~\ref{fig:min_fun} shows the evolution over time of the Minkowski function \cite[\S 3.3]{blanchini1999set} associated to the C-polytope $\mc{S}$, formally defined as $\psi_{\mc{S}}(x_F) \coloneqq \mathrm{min}_{\lambda \geq 0} \, \{\lambda \mid x_F \in \lambda \mc{S} \}$. By randomly drawing $10^3$ initial points in $\mc{S}$, we compute $\psi_{\mc{S}}(x_F(t))$, where $x_F(t)$ is the closed-loop trajectory originating from each initial state with control law in \eqref{eq:vertex_law} and admissible inputs $\{C^\star_{i,600} \bar{w} + d^\star_{i,600}\}_{i \in \mc{V}}$. The fact that $\psi_{\mc{S}}(x_F(t)) \leq 1$, for all $t \in \N$ and initial condition, indicates that $\mc{S}$ is not only invariant\footnote{This would suffice for any $x_F(0) \in \mathrm{vert}(\mc{S})$, where $\psi_{\mc{S}}(x_F(0)) = 1$.}, but also a contractive set for the dynamics in \eqref{eq:LTI_network}.

\section{Conclusion and outlook}
By combining results in system theory and the scenario approach, we provide out-of-sample certificates on the controlled invariance property of a given set with respect to a black-box \gls{LTI} system whose nominal parameters may not be determined with certainty.
We propose a data-based sampling procedure to select feasible inputs at the vertices of the given set, which allows us to verify the controlled invariance property of such a set through an \gls{LP}. If the \gls{LP} is feasible, we establish probabilistic bounds on the controlled invariance property of the given set \gls{wrt} the nominal \gls{LTI} system.

Directions for future work include considering different sampling policies and extending the controlled invariance property verification of given sets \gls{wrt} broader classes of systems, such as linear systems with polytopic uncertainty.

\appendix
\emph{Proof of Lemma~\ref{lemma:singleLP_feasibility}:}
	The Kronecker product  in the matrix $G$ in \eqref{eq:linear_prog_affine} induces a decoupled structure that allows us to focus on a single vertex $i \in \mc{V}$ at a time: the generalization to the entire set $\mc{S}$ follows readily.
	For some $v \in \mathrm{vert}(\mc{S})$, consider $G = \col(H, F B(\delta)) \in \R^{(q+p) \times m}$ and $l = \col(\bsone_{q}, \bsone_{p} - FA(\delta)v) \in \R^{q+p}$. Since $H$ and $F$ are full column rank matrices, we also have $\rank(G) = m$, as $m < q+p$, and the vertical concatenation does not alter the rank (note that $\rank(FB(\delta)) \leq m$, as $n \geq m$). From \cite{chernikov1968linear}, a system of inequalities $G u \leq l$, with $\mathrm{rank}(G) = m$, admits a solution if and only if $G$ has a minor $\theta_m = \mathrm{det}(G_\mc{I}) \neq 0$ of order $m$, with $G_\mc{I} \in \R^{m \times m}$ being full rank submatrix of $G$ with row indices $\mc{I} \subset \{1, \ldots, q+p\} \eqqcolon \mc{A}$, such that
	\begin{equation}\label{eq:iff_generic}
		-\frac{1}{\theta_m} \det\left(
		\left[\begin{array}{c|c}
			G_\mc{I} & l_\mc{I}\\
			\hline
			(G)_k  & l_k
		\end{array}
		\right]
		\right) \leq 0, \, \forall k \in  \mc{A} \setminus \mc{I}.
	\end{equation}
	Since $G_\mc{I}$ is a full rank matrix, the determinant of the augmented matrix in \eqref{eq:iff_generic} can be rewritten as $\det(G_\mc{I}) \times \det(l_k - (G)_k G^{-1}_\mc{I} l_\mc{I})$ \cite{horn2012matrix}. This then implies that \eqref{eq:iff_generic} amounts to verify $(G)_k G^{-1}_\mc{I} l_\mc{I} \leq l_k, \, \forall k \in  \mc{A} \setminus \mc{I}$.
	Note that such inequalities guarantee the existence of some $u^\star$ that solves $G u^\star \leq l$. In case $K = 1$, this is equivalent to guaranteeing the existence of some pair $(C, d)$ satisfying $G (C \delta + d) \leq l$, since $C = 0$ and $d = u^\star$ is always a feasible solution. This consideration holds for each $v \in \mathrm{vert}(\mc{S})$, as $\bs{C} \coloneqq \col((C_i)_{i \in \mc{V}})$ and $\bs{d} \coloneqq \col((d_i)_{i \in \mc{V}})$.
	Therefore, in view of the structure of $G$, we can rewrite the set of row indices as $\mc{I} \coloneqq \mc{Q} \cup \mc{P}$, with $\mc{Q} \subset \{1,\ldots,q\}$ and $\mc{P} \subset \{1,\ldots,p\}$. Then, $G_\mc{I} = \col(H_\mc{Q}, FB_{\mc{P}})$ and $l_\mc{I} = \col(\bsone_{|\mc{Q}|}, \bsone_{|\mc{P}|} - (FA(\delta)v)_\mc{P})$, for any $v \in \mathrm{vert}(\mc{S})$. Finally, the conditions in \eqref{eq:iff_conditions} follow by splitting inequalities $(G)_k G^{-1}_\mc{I} l_\mc{I} \leq l_k, \, \forall k \in  \mc{A} \setminus \mc{I}$, between the two sets $\mc{Q}$ and $\mc{P}$, and noting that $(G)_k = (H)_k$ and $l_k = 1$ for any $k \in \{1,\ldots,q\} \setminus \mc{Q}$, while $(G)_k = (FB(\delta))_k$ and $l_k = 1 - (FA(\delta)v)_k$), for any $k \in \{1,\ldots,p\} \setminus \mc{P}$
	\hfill$\blacksquare$

\smallskip

\emph{Proof of Proposition~\ref{prop:feasibility_LP}:}
		With $\bs{\tilde{u}} \coloneqq \col((\bs{u}_j)_{j \in \mc{K}}) \in \R^{mNK}$, a solution to $\diag((G(\delta^{(j)}))_{j \in \mc{K}}) \bs{\tilde{u}}^\star \leq \col((l(\delta^{(j)}))_{j \in \mc{K}})$ exists if one can find an individual $\bs{u}^\star_j \in \R^{mN}$ for each \gls{LP} in \eqref{eq:linear_prog_affine}. Then, the proof follows the same considerations adopted in the one for Lemma~\ref{lemma:singleLP_feasibility}, for each sample $\delta^{(j)} \in \delta_K$.
	\hfill$\blacksquare$


\balance
\bibliographystyle{IEEEtran}
\bibliography{21_LCSS_probstab}

\end{document}